\documentclass[reqno, 11pt]{amsart}
\usepackage{amsmath,amsfonts,amssymb,amsthm}
\usepackage{mathtools}
\usepackage{color}
\usepackage{geometry}
\usepackage{enumitem}
\usepackage[colorlinks]{hyperref}

\newtheorem{Question}{Question}
\newtheorem{lem}{Lemma}[section]
\newtheorem{teo}[lem]{Theorem}
\newtheorem{pro}[lem]{Proposition}
\newtheorem{cor}[lem]{Corollary}

\newtheorem*{rem*}{Remark}
\newtheorem*{teo*}{Theorem}

\theoremstyle{definition}
\newtheorem{definition}[lem]{Definition}

\theoremstyle{remark}
\newtheorem{rem}[lem]{Remark}

\newcounter{claimcounter}
\numberwithin{claimcounter}{lem}

\usepackage{scalerel,stackengine}
\stackMath
\newcommand\reallywidehat[1]{%
	\savestack{\tmpbox}{\stretchto{%
			\scaleto{%
				\scalerel*[\widthof{\ensuremath{#1}}]{\kern-.6pt\bigwedge\kern-.6pt}%
				{\rule[-\textheight/2]{1ex}{\textheight}}%WIDTH-LIMITED BIG WEDGE
			}{\textheight}% 
		}{0.5ex}}%
	\stackon[1pt]{#1}{\tmpbox}%
}

\DeclareMathOperator{\Id}{Id}

\DeclareMathOperator{\lcm}{lcm}

\DeclareMathOperator{\im}{Im}

\DeclareMathOperator{\Hom}{Hom}

\DeclareMathOperator{\Sur}{Epi}

\DeclareMathOperator{\Aut}{Aut}
\DeclareMathOperator{\Ann}{Ann}

\DeclareMathOperator{\iid}{id}

\DeclareMathOperator{\Tor}{Tor}

\newcommand{\Z}{\mathbb{Z}}
\newcommand{\F}{\mathbb{F}}

\newcommand{\N}{\mathbb{N}}

\newcommand{\G}{\mathcal{G}}

\newcommand{\todo}[1]{}
\renewcommand{\todo}[1]{{\color{red} TODO: {#1}}}

\title{Free factors and  profinite completions}

\author{Alejandra Garrido}

\author{Andrei Jaikin-Zapirain}
 
\address{Departamento de Matem\'aticas, Universidad Aut\'onoma de Madrid \and  Instituto de Ciencias Matem\'aticas, CSIC-UAM-UC3M-UCM}
\email{alejandra.garrido@uam.es}
\email{andrei.jaikin@uam.es}

\begin{document}

\begin{abstract}
	Can one detect free products of groups via their profinite completions? We answer positively among virtually free groups.
	 More precisely, we prove that a subgroup of a finitely generated virtually free group $G$ is a free factor if and only if its closure in the profinite completion of $G$ is a profinite free factor.
	This generalises results by Parzanchevski and Puder for free groups that were  later also proved by Wilton. 
	Our methods are entirely different to theirs, combining homological properties of profinite groups and the decomposition theory of Dicks and Dunwoody.
%	\todo{put this elsewhere in paper (not so optimist)}{\color{red}This points to a theoretical framework within which the general question above may be tackled.  }

% Let $F$ be a finitely generated  free group. Puder and Parzanchevski proved that  the values of a  tuple $(w_1,\ldots, w_k)$ of elements of $F$  are uniformly distributed in any finite group if and only if   $(w_1,\ldots, w_k)$ form part of a free generating tuple of $F$. Equivalently,  if $H$ is a finitely generated subgroup of $F$, then  $H$ is a free factor of $F$ if and only if  the closure $\overline H$ of $H$ in $\widehat F$ is a free profinite factor of $\widehat F$. In this paper we show that the same claim holds also for finitely generated  virtually free groups. %The proof uses the theory of ends introduced by Stallings and 
%

%Let $G$ be a finitely generated virtually free group. We show that if $H$ is a finitely generated subgroup of $G$, then $H$ is a free factor of $G$ if and only if the closure $\overline H$ of $H$ in the profinite completion $\widehat G$  of $G$ is a free profinite factor of $\widehat G$. This generalizes  a result of Puder and Parzanchevski, who proved this statement in the free case. 
\end{abstract}

\maketitle

\section{Introduction}
One of the most natural and successful ways to study an infinite object is via finite approximations to it. 
A natural way to approximate an infinite group is to consider larger and larger finite quotients of it. 
The object that encodes all finite quotients of a group $G$ is its profinite completion $\widehat{G}$. Indeed, this is the categorical limit of all finite quotients and the natural maps between them. 
In order for $\widehat{G}$ to be a faithful approximation to $G$, the latter must be residually finite, in which case $G$ embeds in $\widehat{G}$.
We restrict ourselves to the study of these groups. 

One then may wonder to what extent a residually finite group is determined by its profinite completion. 
The best possible answer is known as \emph{absolute profinite rigidity}: if $\widehat{G}\cong \widehat{H}$ then $G\cong H$. 
It is straightforward to see that finite groups and finitely generated abelian groups are absolutely profinitely rigid. 
Beyond that, things get difficult very quickly. 
There are examples of virtually cyclic groups  (\cite{Baumslag}) and of  finitely generated torsion-free nilpotent groups of class 2 (\cite{GrunewaldScharlau}) that are not absolutely profinitely rigid. 

Perhaps the boldest question in this area is that posed by Remeslennikov (\cite[Problem 5.48]{Kourovka}): \emph{Are free non-abelian groups absolutely profinitely rigid?
}

The answer to this question still seems remote, given current knowledge.  
Nevertheless, some progress has been made on variants of this question: free groups and surface groups are determined, among limit groups,  by their profinite completion (\cite[Corollary D]{Wi18}, \cite[Theorem 2]{Wilton_comptes}); certain %(non-free)
 Kleinian groups and triangle groups are absolutely profinitely rigid (\cite{BMRS18, BMRS20}).

A natural way to attack Remeslennikov's question is to ask whether one can detect a free factor of a group from its profinite completion.
It is an exercise in the definition of free profinite products (\cite[9.1.1]{RibesZalesskii}) to show that if $G$ is the free product of subgroups $H, K$ then $\widehat{G}$ is the free profinite product of the closures $\overline{H},\overline{K}$. 
One might hope that the converse is also true:

\begin{Question}\label{qu:recognise_free_factors_profinitely}
	Let $G$ be a finitely generated residually finite group and $H\leq G$. 
	If $\overline{H}$ is a free profinite factor of $\widehat{G}$, must $H$ be a free factor of $G$?
\end{Question}

This question is quite tricky.
The only results so far answering Question \ref{qu:recognise_free_factors_profinitely} require $G$ to be a free group, and then the answer is ``yes''. 
This case was first shown by Parzanchevski and Puder \cite{PP15} and later, using different methods, by  Wilton \cite{Wi18}. 
%
%\todo{Puder Parzan did it first, Wilton reproved (rephrase so it is clear)}
%
Each of the two proofs rely on beautiful and deep results. 
The first uses detailed calculations for the expected  number of  common fixed points of the image of a  finitely generated $H$ under a uniformly random homomorphism $G\rightarrow \mathrm{Sym}(n)$, for each $n$.   
In particular, Puder and Parzanchevski obtain that if $w\in G$ is not a primitive word then, for large enough $n$, the average number of fixed points of the image of $w$ under a uniformly random   homomorphism $G\rightarrow \mathrm{Sym}(n)$ is strictly larger than $1$. 
Wilton's result is a consequence of another result of his that  partially answers a question of Gromov: if the fundamental group of a non-trivial graph of free groups with
cyclic edge groups is one-ended and hyperbolic then it contains a  surface subgroup. 
%Given a finitely generated  subgroup $H$ of a finitely generated free group $G$, Wilton shows that if $\overline{H}$ is a free factor of $\hat{G}$ then the double $G\ast_{H}G$ is either free or contains a surface group, and finds a contradiction if the latter holds, using that surface groups have nontrivial second cohomology which is preserved when passing to profinite completions.   

Here we use entirely different methods to answer Question \ref{qu:recognise_free_factors_profinitely} positively when $G$ is virtually free. 

We need some notation.
Let $H$ be a subgroup of a finitely generated group $G$ and $\gamma:H\rightarrow P$ a homomorphism to a finite group $P$. 
If $H$ is a free factor of $G$, then the number $h(G,H,\gamma,P)$ of homomorphisms $\tilde{\gamma}:G\rightarrow P$ that extend  $\gamma$  is independent of the choice of $\gamma$; it is simply the number of homomorphisms from the free complement of $H$ to $P$. 
It turns out that the converse is also true when $G$ is virtually free.

\begin{teo}\label{teo:vfree}
	Let $G$ be a finitely generated virtually free group. Let $H$ be a finitely generated   subgroup of $G$ and $\overline H$ the closure of $H$ in  $\widehat G$. 
	The following are equivalent:
	
	\begin{enumerate}[label=\textnormal{(\alph*)}]
		\item $H$ is a free factor of $G$;
		\item for every  finite group $P$, the number $h(G,H,\gamma,P)$ is constant on  $\Hom (H, P)$;
		\item    $\overline H$ is a free profinite factor of $\widehat G$.
		
	\end{enumerate}
\end{teo}

%Our methods are entirely different to those that have appeared previously and seem more readily expansible to other classes of groups $G$. 
We have already explained why (a)$\implies$ (b).
 Now notice that  $\Hom(G,P)$ can be identified with the continuous morphisms $\Hom_{\text{cont.}}(\widehat{G},P)$,  for every finite group $P$. 
If $\widehat{H}=\overline{H}$ then $\Hom(H,P)$ can also be identified with $\Hom_{\text{cont.}}(\overline{H},P)$. 
In this case condition (b) is equivalent to

\smallskip

(b') for every finite group $P$, the number  $h(\widehat{G},\overline{H},\gamma,P)$ is constant on $\Hom_{cont.}(\overline{H},P)$. 

\smallskip

As a first step we show in Proposition \ref{prop:profinitefactor} that (b') is equivalent to (c). 
This is a completely general result.
 The only condition in passing from (b) to (b') is that $H$ be closed in the profinite topology of $G$. This is ensured, for instance, if $G$ is subgroup separable (every finitely generated subgroup is closed in the profinite topology),  also termed LERF. 
 
 The meat of the paper is the proof of (c)$\implies$ (a). 
 Condition (c) implies that there is a non-inner derivation (or crossed-homomorphism) $d:\widehat{G}\rightarrow k[[\widehat{G}]]$ whose kernel is precisely $\overline{H}$, where $k$ is a finite commutative ring. 
 This, and the related definitions, are explained in Section \ref{sec:prelims}. 
 The analogous statement in the discrete setting is almost, but not quite, equivalent to (a). 
 In fact, a result of Dicks following Dunwoody  \cite{Di81} states that if there is a derivation $d':G\rightarrow M$ to a projective $k[G]$-module $M$ with kernel $H$, then $G$ splits as a fundamental group of a finite graph of groups, one of whose vertex groups is $H$, and all of whose edge groups are finite, with conditions on their sizes related to the ring $k$. 
 We must therefore do two things: first obtain some $d'$ from $d$, and second show that the graph of groups decomposition of $G$ actually makes $H$ a free factor. 
 The second thing is done in Section \ref{sec:Dicksproof}. 
 The first thing is done in Section \ref{sec:limits_projective_modules}, where it is shown (Corollary \ref{locvfree}) that, for $G$ virtually free, and $k=\Z/n\Z$ with square-free $n$ the left $k[G]$-module $k[[\widehat G]]$  is isomorphic to a direct union of projective $k[G]$-modules.
 This means, as $G$ is finitely generated, that the image $d(G)$ is contained in a finitely generated projective $k[G]$-module. 
 
 In order to obtain Corollary \ref{locvfree}, we give  in Theorem \ref{teo:Rhatflat_iff_Rcoherent} a criterion for the profinite completion $\widehat{R}$ of a ring $R$ satisfying certain conditions to be a flat $R$-module. 
 By a result of \cite{CoFR}, the criterion implies that every finitely generated submodule of $\F_p[[\widehat{F}]]$ is free, where $F$ is a finitely generated free group.    
  The main theorem could be generalised to wider classes of groups if we can answer the following question.
 \begin{Question}
	For which finitely generated groups $G$ is % other than virtually free and virtually polycyclic groups,  
	the profinite completion of $\F_p[G]$ a subgroup of a  direct union of projective  $\F_p[G]$-modules? 
\end{Question}

It is also natural to ask a different version of our Question \ref{qu:recognise_free_factors_profinitely}, namely:

\begin{Question}\label{ZQ}
	Let $G$ be a finitely generated residually finite group. Assume that $\widehat G$  decomposes as a profinite free product. Does $G$ decompose as a free product?
\end{Question}

After completing this manuscript we learned from P. Zalesskii (\cite[Corollary 4.4]{Zalesskii_vfree}) that this question  has positive answer for certain classes of accessible groups; in particular, for virtually free groups.
However,  we may not conclude that the factors in the free decompositions are related to each other (in particular, it is not clear that any of the free profinite factors is the closure of any of the abstract free factors), because there is no analogue of the Grushko-Neumann theorem in profinite groups (\cite{Lucchini}).

Putting together Theorem \ref{teo:Rhatflat_iff_Rcoherent} and deep results on virtually polycyclic groups (\cite{Jategaonkar,Roseblade73, Roseblade}) yields that the completed group ring $\F_p[[\widehat{G}]]$ of a virtually polycyclic group $G$ is a  flat $\F_p[G]$-module. 
Although this is unrelated to the rest of the paper, it seems an interesting phenomenon in its own right.
\begin{Question}
	For which finitely generated groups $G$ is % other than virtually free and virtually polycyclic groups,  
	the profinite completion of $\F_p[G]$ a flat $\F_p[G]$-module? 
\end{Question}

\section*{Acknowledgments}
The authors are grateful to the anonymous and very efficient referee. 
They also acknowledge partial financial support from Spain's Ministry of Science and Innovation, through grant  PID2020-114032GB-I00 and the Severo Ochoa Programme for Centres of Excellence in Research and Development (CEX2019-000904-S).

 \section{Preliminaries}\label{sec:prelims}

   All rings in this paper have a unit, denoted 1,  and all ring homomorphisms send  the unit to the unit. 
   By a module we will mean a left module if we do not say the opposite.

    A nonempty class of finite groups $\mathcal C$ is
   a {\bf pseudovariety}  if it is closed under taking subgroups, homomorphic images and
   finite direct products.  
    The pro-$\mathcal C$ completion of a group $G$ is denoted by ${G}_{\widehat {\mathcal C}}$ and by   $\widehat{G}$ we denote the profinite completion of $G$.
     The closure of a subgroup $H\leq G$ inside $\widehat{G}$ is denoted $\overline{H}$.
   
   The abstract free product of groups $G_1,G_2$ is denoted by $G_1\ast G_2$, while $G_1\coprod_{\mathcal C} G_1$ denotes the free pro-$\mathcal C$ product of two pro-$\mathcal C$ groups $G_1, G_2$. 
  Recall that $G_1\coprod_{\mathcal C} G_2$ is characterized up to isomorphism by the following universal property: there are continuous homomorphisms $\varphi_i:G_i\rightarrow G_1\coprod_{\mathcal C} G_2$ ($i=1,2$),  and for any pair of continuous homomorphisms $\psi_i: G_i\rightarrow P$ ($i=1,2$) to a finite group $P$ in $\mathcal{C}$ there is a unique continuous homomorphism $\psi:G_1\coprod_{\mathcal C} G_2 \rightarrow P$ such that $\psi\circ\varphi_i=\psi_i$, $i=1,2$. 
   	The basic properties of free pro-$\mathcal C$ products can be  found in \cite[Section 9]{RibesZalesskii}. 
   If $\mathcal C$ is the pseudovariety of all finite groups, then we write $A\coprod B$ instead of $A\coprod_{\mathcal C} B$.
   Let $G$ be a pro-$\mathcal C$-group and $H$ a closed subgroup of $G$. We say that $H$ is a {\bf pro-$\mathcal C$ free factor} of $G$ if there exists a closed subgroup $K$ of $G$ such that  $G= H\coprod_{\mathcal C} K$.

    If $k$ is a ring and $G$ is a group, we denote by $k[G]$ the group ring of $G$ over $k$.
    The {\bf augmentation ideal} of $k[G]$ is the kernel of the canonical map $k[G]\to k$, and it is denoted by $I_{k[G]}$.
     
   %  If $G$ is a pro-$\mathcal{C}$ group and $k$ is a finite ring whose additive group is in $\mathcal{C}$,  the {\bf completed group ring} $k[[G]]$ of $G$ over $k$ is  the inverse limit of $k[G/U]$, where $U$ is a normal  open subgroup of $G$ (in the pro-$\mathcal{C}$ topology).
  %   We suppress the pseudovariety $\mathcal{C}$ from the notation as it will be clear from context. 
   If $G$ is a profinite group and $k$ is a finite ring the {\bf completed group ring} $k[[G]]$ of $G$ over $k$ is  the inverse limit of $k[G/U]$, where $U$ is a normal  open subgroup of $G$.
   If $G=\widehat{H}$ is the profinite completion of a group $H$, then $k[[G]]=k[[\widehat{H}]]\cong \widehat{k[H]}=\varprojlim k[H/N]$ where $N$ ranges through the normal subgroups of finite index in $H$. 
  The {\bf augmentation ideal} of $k[[G]]$ is the kernel of the canonical map $k[[G]]\to k$, and it is denoted by $I_{k[[G]]}$.
     It is a free $k$-module on the profinite space $\{g-1 \mid g\in G\}$ and, if $T$ is a profinite space generating $G$ and containing the identity, then $I_{k[[G]]}$ is generated as a profinite  $k[[G]]-$module by $\{t-1 \mid t\in T\}$ (see \cite[6.3.2]{RibesZalesskii}).
      If $A$ is a closed subgroup of $G$ we denote by $I_{k[[A]]}^G$ the left closed ideal of $I_{k[[G]]}$ generated by $I_{k[[A]]}$.

	Let $G$ be a group, $k$ a ring and $M$ a $k[G]$-module. 
	A \textbf{derivation} or \textbf{crossed homomorphism} is a map $d:G\rightarrow M$ such that $d(xy)=xd(y)+d(x)$ for all $x,y\in G$. 
	The derivation is \textbf{inner} if there is some $m\in M$ such that $d=\mathrm{ad_m}: x\rightarrow (1-x)m$. 
%	The first cohomology group of $G$ with coefficients in $M$ is the quotient of all derivations from $G$ to $M$ by the normal subgroup of all inner derivations. 

	Let $\Gamma$ be a profinite group, $k$ a profinite ring and $A$ a discrete  $k[[\Gamma]]$-module. 
	A \textbf{continuous derivation} $d:\Gamma\rightarrow A$ is a continuous map that satisfies $d(gh)=gd(h)+d(g)$ for all $g,h\in\Gamma$. 
	Denote by $\mathrm{Der}(\Gamma, A)$ the set of continuous derivations from $\Gamma$ to $A$.

The following Lemma is well-known to experts, but we could not find a reference with the exact statement. 
It uses similar ideas to \cite[Proposition 9.3.8]{RibesZalesskii}.

\begin{lem}\label{lem:freeprofprod_implies_idealsdecompose} 
Let $k$ be  a finite ring and $\mathcal C$ an extension-closed pseudovariety of finite groups  containing the group $(k,+)$. Let $\Gamma$ be a pro-$\mathcal C$ group and assume that $\Gamma=\Delta_1\coprod_{\mathcal C} \Delta_2$. 
Then $$I_{k[[\Gamma]]}=I_{k[[\Delta_1]]}^{\Gamma}\oplus I_{k[[\Delta_2]]}^{\Gamma}.$$
\end{lem}

\begin{proof}
	To simplify notation, write $I:=I_{k[[\Gamma]]}$ and $J_i:=I_{k[[\Delta_i]]}^{\Gamma}$, for $i=1,2$.
	Since $\Gamma$ is generated by $\Delta_1$ and $\Delta_2$, the augmentation ideal $I$ is generated as a profinite $k[[\Gamma]]$-module by $\{\delta-1 \mid \delta \in \Delta_1 \cup \Delta_2\}$. 
	So evaluating in $I$  gives a canonical surjective $k[[\Gamma]]$-morphism $\pi: J_1\oplus J_2 \rightarrow I$.
	To prove the claim, we will find an inverse  $\rho:I\rightarrow J_1\oplus J_2$ of $\pi$. 
	
	For $i=1,2$ and each open normal subgroup $U\unlhd_o\Gamma$, denote by $J_{i,U}:=I_{k[[\Delta_iU/U]]}^{\Gamma/U}$ the image of $J_i$  in $k[\Gamma/U]$. 
	Note  that $J_1\oplus J_2=\varprojlim_{U\unlhd_o\Gamma}J_{1,U}\oplus J_{2,U}$. 
	Since $\Gamma$ acts on   $J_{1,U}\oplus J_{2,U} \in \mathcal{C}$, the group $H_U:=(J_{1,U}\oplus J_{2,U})\rtimes \Gamma$ is a pro-$\mathcal C$ group. 
	For $i=1,2$, define $\phi_{i,U}:\Delta_i\rightarrow H_U$ by $\delta \mapsto (\delta U- U, \delta )$ where $\delta U-U \in J_{i,U}$. 
	By definition of $J_{i,U}$, these are continuous homomorphisms. 
	The universal property of free pro-$\mathcal C$ products yields a unique continuous homomorphism $\psi_U:\Gamma \rightarrow H_U$ extending $\phi_{1,U}$ and $\phi_{2,U}$. 
	Denote, respectively,  by $p_J$ and $p_{\Gamma}$ the projections of $H_U$ to $J_{1,U}\oplus J_{2,U}$ and $\Gamma$.
	Since $p_{\Gamma}\circ\phi_{i,U}=\iid_{\Delta_i}$, the universal property of free pro-$\mathcal C$ products applied to $\iid_{\Gamma}$ ensures that $p_{\Gamma}\circ\psi_U=\iid_{\Gamma}$.
	Now define $\rho_U:I\rightarrow J_{1,U}\oplus J_{2,U}$ by $\gamma -1 \mapsto p_J(\psi_U(\gamma))$ for $\gamma \in \Gamma$ and extending by $k$-linearity. 
	This is indeed a profinite $k[[\Gamma]]$-module morphism because, as $\psi_U$ is a homomorphism,
	 for every $\gamma, \epsilon\in \Gamma$ we have
	  \begin{align*} 
	   \rho_U(\epsilon(\gamma -1))&=\rho_U(\epsilon\gamma-1 - (\epsilon -1)) =p_J(\psi_U(\epsilon\gamma))-p_J(\psi_U(\epsilon))\\
	  &=p_{\Gamma}(\psi_U(\epsilon))p_{J}(\psi_U(\gamma))+p_J(\psi_U(\epsilon))-p_J(\psi_U(\epsilon)) \\
	  &=\epsilon p_J(\psi_U(\gamma))	  =\epsilon \rho_U(\gamma).
	  \end{align*}
	  
	 Let $\rho:I\rightarrow J_1\oplus J_2$ be the  map given by the universal property of inverse limits for $J, \{J_{1,U}\oplus J_{2,U} \mid U\unlhd_o\Gamma\}$ and $\{\rho_U\mid U\unlhd_o \Gamma\}$. 
	 Observe that, for $i=1,2,$ and every $U\unlhd_o\Gamma$ and $\delta \in \Delta_i$, we have $\rho_U(\pi(\delta-1))=\rho(\delta-1)=\delta U-U$, so $\rho_U\circ\pi$ is the projection map $J_1\oplus J_2 \rightarrow J_{1,U}\oplus J_{2,U}$ and thus $\rho\circ \pi=\iid_{\Gamma}$, as required. %(by the uniqueness of the map given by the universal property of inverse limits)
\end{proof}

Recall that a \textbf{supernatural number} is a formal product 
$n=\prod_{p}p^{n(p)}$ where $p$ runs through all prime numbers and $n(p)\in \N\cup \{\infty\}$. 
By convention, $n<\infty$ and $\infty+n=n+\infty=\infty+\infty=\infty$ for all $n\in\N$. 
A supernatural number $m=\prod_p p^{m(p)}$ is said to \textbf{divide} another one $n=\prod_p p^{n(p)}$ if $m(p)\leq n(p)$ for every prime $p$. 
The \textbf{lowest common multiple}  of a collection of supernatural numbers $\{n_i=\prod_{p}p^{n(i,p)} \mid i\in I\}$ is defined as $\lcm\{n_i : i\in I\}:=\prod_p p^{n(p)}$ where $n(p)=\max\{n(i,p) \mid i\in I\}$.
If $H$ is a closed subgroup of a profinite group $\Gamma$, we put $|\Gamma:H|=\lcm\{|\Gamma:HN|\colon N\unlhd_o \Gamma\} $.

Let $\Gamma$ be a profinite group with closed subgroups $A,B\leq \Gamma$ and $p$ a prime number.
The subgroup $A$ acts continuously on  $\F_p[[\Gamma/B]]=\varprojlim_{N\unlhd_o\Gamma}\F_p[\Gamma/BN]$  by left coset multiplication: $am=a(m_N)_{N\unlhd_o\Gamma}=(am_N)_{N\unlhd_o\Gamma}$ where $$am_N=a\left(\sum_{gBN\in\Gamma/BN}f_{N,g}gBN\right)=\sum_{gBN\in\Gamma/BN}f_{N,g}agBN, \quad f_{N,g}\in \F_p.$$
The above then means that, for each $N\unlhd_o \Gamma$, $\F_p[\Gamma/BN]$ decomposes as an $A$-module into $$\F_p[\Gamma/BN]\cong \prod_{AgBN\in A\backslash\Gamma/B}\F_p[AgBN]\cong\prod_{AgBN\in A\backslash \Gamma/B}\F_p[A/A\cap B^gN] $$
(since we are using left cosets, conjugation is done on the left: $B^g=gBg^{-1}$).
Taking the inverse limit, we obtain that the $A$-module (and even the $\F_p[[A]]$-module) $\F_p[[\Gamma/B]]$ decomposes as 
\begin{equation}\label{eq:G/B as prod of A orbits}
\F_p[[\Gamma/B]]\cong \prod_{AgB\in A\backslash \Gamma / B}  \F_p[[A/A\cap B^g]].
\end{equation}

Denote by $\F_p[[\Gamma/B]]^A$ the set $\{m\in \F_p[[\Gamma/B]]:\ am=m \textrm{\ for every\ }a\in A\}$ of fixed points of $A$. 
By \eqref{eq:G/B as prod of A orbits},  this set is reduced to 0 if and only if $\F_p[[A/A\cap B^g]]^A=\{0\}$ for every $g\in \Gamma$. 
This observation will help in proving the following.

\begin{lem}\label{lem:no_fixed_points_modH}Let $\Gamma$ be a profinite group with closed subgroups $A,B\leq \Gamma$ and $p$ a prime number.
Then $\F_p[[\Gamma/B]]^A=\{0\}$ if and only if 
 $p^{\infty}$ divides $  |A:A\cap B^g |  $
 for every $g\in \Gamma$. 
\end{lem}

In order to prove the lemma we need the following result on finite groups.

\begin{lem}\label{finitefix}
Let $G$ be a finite group, $H\leq G$ and $N\unlhd G$. 
The image of the map
$\F_p[G/H]^G\to \F_p[G/NH]^G$ is trivial if and only if $p$ divides $|NH:H|$.
\end{lem}
\begin{proof} 
	Observe that every element of $\F_p[G/H]^G$ is a  constant function $\displaystyle{f\sum_{gH\in G/H} gH}$ for some $ f\in \F_p$ and that its image in 
   $\F_p[G/NH]^G$ is $\displaystyle{\sum_{gNH\in G/NH}(|NH:H|f)gNH}$.
\end{proof}

\begin{proof}[Proof of Lemma \ref{lem:no_fixed_points_modH}]
By the observation made before the statement, we need only consider $A/A\cap B^g$ for some $g\in \Gamma$ and, since $A\cap B^g$ is simply a subgroup of $A$, we can reduce to the case where $A=\Gamma$.
For ease of exposition, assume also that $\Gamma$ is countably based, so that $\F_p[[\Gamma/B]]=\varprojlim_{i\in\N}\F_p[\Gamma/BN_i]$ where $\Gamma = N_0\geq N_1\geq \dots$ is a countable chain of open normal subgroups such that $\bigcap_{i\in\N}N_i=\{1\}$. 
The general case is not difficult, just more tedious to notate.

Assume first that $p^{\infty}$ divides $|\Gamma:B|$.
	Let $m=(m_i)_{i\in\N} \in \F_p[[\Gamma/B]]^\Gamma$,  so  $m_i\in  \F_p[[\Gamma/BN_i]]^\Gamma$
	for each $i\in\N$.
   Since   $p^{\infty}$ divides $|\Gamma:B|$, $p$ divides $|BN_i:BN_{i+1}|$ for infinitely many values of $i$. By Lemma \ref{finitefix},  for these $i$, $m_i=0$. Thus, $m=0$.
	
Conversely, if the $p$-part of  $|\Gamma:B|$ is finite, there must be some $k$ such that $p$ does not divide $|BN_k:BN_{i}|$ for all $i\ge k$. 
Put $$m_i= \frac 1{ |BN_k:BN_{i}|}\sum_{gBN_i\in \Gamma/BN_i} gBN_i \in \F_p[\Gamma/BN_i] \textrm{\ and \ } m=(m_i)\in \F_p[[\Gamma/B]].$$ It is clear that $0\ne m\in  \F_p[[\Gamma/B]]^\Gamma$. 
\end{proof}

We finish the section with another auxiliary result that we will need later.

\begin{lem}\label{lem:chain_char_sgps}
	Let $\mathcal{C}$ be a pseudovariety of finite groups and $G$ a finitely generated pro-$\mathcal{C}$ group.
	Then there is a chain $G>N_1>N_2>\dots$ of open characteristic subgroups of $G$ with trivial intersection such that for every $i$ and $\phi_i\in\Sur(G,G/N_i)$ we have that $\ker\phi_i=N_i$. 
\end{lem}
\begin{proof}
	For each $i\in\N$, let $N_i$ be the intersection of all open normal subgroups of $G$ of index at most $i$.
	Since $G$ is finitely generated, there are finitely many such subgroups, so $N_i$ is a closed subgroup of finite index and therefore open. 
	By construction, each $N_i$ is characteristic and they form a descending chain whose intersection is trivial.
	Now, fix $i\in\N$ and note that $G/N_i$ is a subdirect product of all finite quotients of $G$ of size at most $i$.
	Given $\phi\in\Sur(G,G/N_i)$, anything in $\ker\phi$ must have trivial image in each of the finite factors in the subdirect product $G/N_i$,  so $\ker\phi\leq N_i$.
	Conversely, if $g\notin\ker\phi$, then $g$ has nontrivial image in some factor of the subdirect product $G/N_i$, so cannot be contained in $N_i$.
	Thus $\ker\phi=N_i$.
\end{proof}

\section{Free profinite factors}

	Throughout this section, $\mathcal{C}$ denotes an extension-closed pseudovariety of finite groups.
	Subgroups of pro-$\mathcal{C}$ groups are closed and homomorphisms are continuous, unless stated otherwise.

  The main result of this section is a general criterion for  a given closed subgroup $H$ of a pro-$\mathcal C$ group $G$ to be a free pro-$\mathcal{C}$ factor in terms of extending homomorphisms $H\rightarrow P\in \mathcal{C}$ to $G$.
  It is the key to showing that (b) and (c) in Theorem \ref{teo:vfree} are equivalent, as explained in the introduction.

  Recall that given $H\leq G$ and a homomorphism $\gamma:H\rightarrow P$ to a finite group $P\in \mathcal{C}$, we denote by $h(G,H,\gamma,P)$ the number of homomorphisms $\tilde{\gamma}:G\rightarrow P$ such that $\tilde{\gamma}|_H=\gamma$.
   The number of epimorphisms $G\twoheadrightarrow P$ that extend $\gamma$ is denoted by $e(G,H,\gamma,P)$.
  
  \begin{pro}\label{prop:profinitefactor}
  Let $G$ be a finitely generated pro-$\mathcal C$ group and $H$ a closed finitely generated subgroup of $G$.  
   Then $H$ is a free pro-$\mathcal C$ factor of  $G$ if and only if  for every  group $P$ from $\mathcal C$, the number $h(G,H,\gamma,P)$ is independent of the choice of  $\gamma\in\Hom (H, P)$.
  \end{pro}

In the language of \cite{PP15}, the result says that $H$ is a free pro-$\mathcal{C}$ factor of $G$ if and only if $H$ is measure-preserving, meaning that for every $P\in \mathcal{C}$, if $\varphi\in \Hom(G,P)$ is a random homomorphism chosen uniformly, then $\varphi_{|H}$ is uniformly distributed in $\Hom(H,P)$.

 In order to prove Proposition \ref{prop:profinitefactor} we will study whether an isomorphism between closed subgroups of a profinite group can be extended to an automorphism of the ambient group.   
%If $H$ is a closed subgroup of a finitely generated profinite group $G$, $P$ is a finite group and $\gamma\in  \Hom(H, P)$, we denote by  $e_{H\le G}(\gamma)$ the number of epimorphisms $\tilde \gamma\in \Sur(G, P)$  such that $\tilde \gamma_{|H}=\gamma$.   
 
 \begin{pro} \label{prob}
 	 Let $G$ be a finitely generated pro-$\mathcal C$ group and $\alpha:H_1\to H_2$ a continuous isomorphism between two closed subgroups of $G$.  
  Then the following are equivalent.
\begin{enumerate}
\item[(a)] there exists $\tilde \alpha\in \Aut(G)$ such that $\tilde \alpha_{|H_1}=\alpha$;
\item  [(b)] for every  group $P\in \mathcal C$ and every  $\gamma\in \Hom (H_2,P)$, 
$$h(G,H_1,\gamma\circ\alpha,P)=h(G,H_2,\gamma,P);$$
\item[(c)] for every  group $P\in \mathcal C$ and every  $\gamma\in \Hom (H_2,P)$, $$e(G,H_1,\gamma\circ\alpha,P)=e(G,H_2,\gamma,P);$$
 
\item [(d)]  for every  group $P\in \mathcal C$ and every   $\gamma\in \Hom (H_2,P)$,  
$$ \textrm{if $e(G,H_2,\gamma,P)\ne 0$, then $e(G,H_1,\gamma\circ\alpha,P)\ne 0$}.$$\end{enumerate}
\begin{rem*}
Since $G$ is finitely generated, by a well-known result of Nikolov and Segal \cite{NS07}, any homomorphism in $\Hom(G,P)$ is continuous.  Thus, for non-continuous $\gamma$ the conditions   (b), (c)  and (d) hold trivially as all numbers involved are 0.  
 \end{rem*}
\end{pro}
 \begin{proof} The implications $(a)\Longrightarrow (b)$ and $(c)\Longrightarrow (d)$ are immediate.

 For a homomorphism $\beta\in \Hom(H, P)$ and a subgroup $\im \beta \le Q\le P$, we denote by $\beta^Q$ 
 the corestriction of $\beta$ to $Q$; that is, the homomorphism in $\Hom(H, Q)$ such that  $\beta^Q(h)=\beta(h)$ for every $h\in H$.

 The  implication  $(b)\Longrightarrow (c)$ is proved by induction on the order of $P$ and taking into account the equations 
\begin{align*}  
	h(G,H_1,\gamma\circ\alpha,P) &=
 \sum_{\im\gamma\le Q\le P}  e(G,H_1,\gamma^Q\circ \alpha,Q)\\  
 h(G,H_2,\gamma,P ) &=
 \sum_{\im\gamma\le Q\le P}  e(G,H_2,\gamma^Q,Q).
\end{align*}

 Let us prove  the  implication $(d)\Longrightarrow (a)$. 
  By Lemma \ref{lem:chain_char_sgps}, there is a chain $G>N_1>N_2>\dots$ of characteristic open subgroups of $G$ with trivial intersection such that for each $i$, if $\phi_i\in\Sur(G,G/N_i)$ then $\ker \phi_i=N_i$. 
  For each $i\in\N$, put  $P_i=G/N_i$, denote by $\delta_i:G\rightarrow P_i$ the canonical map and by $\gamma_i=(\delta_i)_{|H_2}$ the natural map $H_2\rightarrow P_i$. 
 By assumption, there exists $\phi_i\in\Sur(G,P_i)$ such that $\phi_{|H_1}=\gamma_i\circ \alpha$. 
 Since $\ker\phi_i=N_i$ we obtain an automorphism of $P_i$, denoted $\overline{\phi_i}$, which satisfies that $\overline{\phi_i}\circ(\delta_i)_{|H_1}=\gamma_i\circ\alpha$. 
 This shows that the set $S_i=\{\tau\in\Aut(P_i): \tau\circ(\delta_i)_{|H_1}=\gamma_i\circ\alpha\}$ is not empty. 
 Now, for $j>i$, the maps $\delta_{ji}:P_j\rightarrow P_i$ induce maps $S_j\rightarrow S_i$ that turn $(S_i)_{i\in\N}$ into a directed system of non-empty sets.
Their inverse limit $S=\varprojlim_{i}S_i$
 is therefore not empty and it is easy to check that any $\tilde{\alpha}\in S$ is an automorphism of $G$ such that $\tilde{\alpha}_{|H_1}=\alpha$, as required.
 \end{proof}

%Now we are ready to prove  Proposition \ref{prop:profinitefactor}

   \begin{proof} [Proof of Proposition \ref{prop:profinitefactor}] 
   The ``only if "part is clear. Let us show the ``if " part.
   
   Let us first assume that $H$ is finite.
   Let $P$ be a group from $\mathcal C$ and $\gamma\in \Hom (H, P)$.
    Then, since $h(G,H,\gamma,P)$ is constant on  $\Hom (H, P)$, \begin{equation}\label{eq:numberhoms}
    h(G,H,\gamma,P)\cdot |\Hom(H,P)|=|\Hom(G, P)|.
    \end{equation}
    Let   $\alpha:K\to H$ be  an isomorphism, $\gamma\in \Hom(H,P)$ and  consider the free pro-$\mathcal C$ product $\Gamma=G\coprod_{\mathcal C} K$.
    By the universal property, every homomorphism $\Gamma \rightarrow P$ is determined by two homomorphisms $G\rightarrow P$ and $K\rightarrow P$. 
    This observation gives the first and last equalities in the following:

  \begin{multline*}
  h(\Gamma,H,\gamma,P)=h(G,H,\gamma,P)\cdot |\Hom(K,P)|\\=h(G,H,\gamma,P)\cdot |\Hom(H,P)| \overset{\eqref{eq:numberhoms}}{=} |\Hom(G, P)|=h(\Gamma,K,\gamma\circ\alpha,P).
  \end{multline*}
  By Proposition \ref{prob}, there exists $\tilde \alpha\in \Aut(\Gamma)$ that extends $\alpha$. 
  In particular, there exists a subgroup $G^\prime$ in $\Gamma$ such that $\Gamma=G^\prime\coprod_{\mathcal C} H$. 
  By \cite[Corollary 7.1.3 (a)]{Ribesprofinitetree}, there exists $g\in\Gamma$ such that $K^g\le G^\prime$ or $K^g\le H$. The second possibility cannot occur because the normal subgroup of $\Gamma$ generated by $K$ has trivial intersection with $H$.
   Thus $K^g\le G^\prime$, and so, 
   $G\cong \Gamma/\overline{ \langle K^{\Gamma}\rangle}\cong (G^\prime/\overline{ \langle K^{  G^\prime}\rangle})\coprod_{\mathcal C} H.$ 
  Hence $H$ is a free pro-$\mathcal C$ factor of $G$.
  
  Now assume that $H$  is arbitrary. Let $N$ be an open normal subgroup of $G$. 
 
  We put $$H_N^\prime=H/(N\cap H) \textrm{\  and \ }  G_N^\prime =G/\overline{ \langle (N\cap H)^{  G}\rangle}.$$ We see $H_N^\prime$ as a subgroup of $G_N^\prime$. Let $P\in \mathcal C$ and $\gamma\in \Hom (H_N^\prime, P)$.
   We denote by $\tilde \gamma \in  \Hom (H, P)$ the canonical lift of $\gamma$. 
   Then we have that
  $h(G_N^\prime, H_N^\prime,\gamma,P)=h(G,H,\tilde \gamma,P)$.
   Hence $h(G_N^\prime, H_N^\prime,\gamma,P)$ is constant on  $\Hom (H_N^\prime, P)$.
    Using what we have proved before, $H_N^\prime$ is a free pro-$\mathcal C$ factor of $G^\prime_N$. Let $d$ be the number of generators of $G/\overline{H^G}$. Then we have just proved that the set
    $$S_N=\{(h_1,\ldots,  h_d)\in  (G^\prime_N)^d\colon  G^\prime_N=\overline {\langle h_1,\ldots, h_d\rangle }\coprod_{\mathcal C} H^\prime_N\}$$ is not empty.
    
  Observe that $\{S_N\}$ form an inverse system of compact sets. Thus, there exists $(g_1,\ldots, g_d)\in \varprojlim_N S_N\subset G^d$. It is clear that $G=\overline {\langle g_1,\ldots, g_d\rangle }\coprod_{\mathcal C} H$.

  \end{proof}

\section{Profinitely flat modules}\label{sec:flat_rings}

In this section we establish a connection between  coherence of rings and flatness of profinite completions. In the context of commutative rings the relation between completions, flatness and coherence was studied in \cite{Ha11}.

Recall that  $R$ is called {\bf right (resp. left) coherent} if every finitely generated submodule of a free right (resp. left) $R$-module is finitely presented.
It is \textbf{coherent} if it is both right and left coherent. 
If $M$ is a (right/left) $R$-module, its {\bf profinite completion} $\widehat{M}$ is the (right/left) $R$-module $\varprojlim_{N\leq_f M} M/N$ where $N$ ranges over all submodules of $M$ of finite index (as a subgroup). We denote by $\delta_M : M\to \widehat M$ the natural map from the group to its profinite completion. If $M$ is residually finite we will not distinguish between the elements of $m\in M$ and of $\delta_M(m)\in \widehat M$.
 
 Observe that if $M$ is a left (right) $R$-module, then  $\widehat M$ is a  left (right) $\widehat R$-module:  given  $r=(r_I)$, where $r_I\in R/I$ and $I$ is an ideal of $R$ of finite index, and  $m=(m_N)\in \widehat{M}$, where $m_N\in M/N$ and $N$ is  a submodule of  $M$ of finite index,  define $ r\cdot m=(  r_{\Ann(M/N)}\cdot m_N)$ ($  m\cdot r=(m_N\cdot   r_{\Ann(M/N)})$. In the case $M=R$, the completion $\widehat{R}$ is an $R$-bimodule.

% It looks like this notion has not been studied previously.  However, as we will see, this notion  has an interesting relation with coherence and can be seen as strong form of residual finiteness of $R$.

%Recall that  $R$ is called {\bf left coherent} if every finitely generated submodule of a free left $R$-module is finitely presented.
%
The goal of this section is to prove the following:
\begin{teo}\label{teo:Rhatflat_iff_Rcoherent}
	Let $R$ be a finitely generated residually finite ring such that all finitely presented right $R$-modules are residually finite. 
	The following are equivalent:
	\begin{enumerate}[label=\textnormal{(\arabic*)}]
		\item\label{it:Rhatflat+FPMRF} The profinite completion $\widehat{R}$ of $R$ is a flat left $R$-module. 
		%	\item\label{it:Rhat/R_flat} The right $R$-module $\widehat{R}/R$ is flat. 
		\item\label{it:Rcoherent+FPMRF} $R$ is right coherent.
	\end{enumerate}
\end{teo}

Before proving the theorem, we need some auxiliary results. 
%Since the profinite completion functor on groups is right exact \todo{find reference for this} we obtain the following.
 \begin{lem}[See Proposition 3.2.5 of \cite{RibesZalesskii}] \label{indpro} Let $ A\to B\to C\to 0$ be an exact sequence of right $R$-modules. 
 	Then the induced sequence of right  $\widehat R$-modules
	$ \widehat A\to \widehat B \to \widehat C\to 0$ is also exact.
\end{lem}

\begin{lem} \label{lem:alpha}Let $M$ be a right $R$-module and  $\alpha_M: M\otimes_R\widehat R \to  \widehat M$,  $m\otimes r \mapsto \delta_M(m)\cdot r$  the natural map of right $\widehat R$-modules.
	\begin{enumerate}
		\item If $M$ is finitely generated, then $\alpha_M$ is onto.
		\item If $M$ is finitely presented, then $\alpha_M$ is an isomorphism.
		\item A (right) $R$-module morphism $\gamma: M_1\to M_2$ induces the following commutative diagram
		$$\begin{array}{ccc}
		M_1\otimes_R \widehat R &\xrightarrow{\gamma\otimes \Id} & M_2\otimes_R\widehat R\\
		{\Big\downarrow}{\alpha_{M_1}} &&{\Big\downarrow}{\alpha_{M_2}}\\
		\widehat{M_1} &\xrightarrow{\widehat \gamma} &\widehat {M_2}\end{array} .$$
		%where $\im\widehat{\gamma}$ is the closure of $\im \gamma$ in $\widehat{M_2}$. 
	\end{enumerate}
	
\end{lem}
\begin{proof}
	Suppose that $M$ is finitely generated, say $M=m_1R + \dots + m_dR$.
	Then $\alpha_M(M\otimes 1)=\delta_M(M)$ and $\alpha_M(M \otimes_R \widehat{R})=\delta_M(m_1)\widehat R + \dots + \delta_M(m_d)\widehat{R}$. 
	As each $\delta_M(m_i)\widehat{R}$ is closed in $\widehat M$, their sum is a closed subset of $\widehat M$ that contains  $\delta_M(M)$, so must be all of $\widehat M$. 
	
	Let $M$ be finitely presented, with free presentation $R^e\rightarrow R^d \rightarrow M \rightarrow 0$ and $e, d$ finite.
	Tensoring with $\widehat R$ and using Lemma \ref{indpro} we obtain the following commutative diagram with exact top and bottom rows:
	$$\begin{array}{cccccc}
	 R^e\otimes_R \widehat{R}& \rightarrow &  R^d\otimes_R \widehat{R} &\rightarrow & M\otimes_R \widehat{R}&\to 0\\
	{\big\downarrow}&& {\big\downarrow}&&{\Big\downarrow}{\scriptstyle\alpha_M}&\\
	\widehat{R}^e&\rightarrow  &\widehat{R}^d&\rightarrow & \widehat M&\to 0
	\end{array}.$$
	The two left downward arrows are isomorphisms, because tensor products and completions commute with finite direct sums.
	A diagram chase shows that $\alpha_M$ is also an isomorphism. 
	
	The last item is equivalent  to $\widehat \gamma$ being a  $\widehat R$-homomorphism, which  is clear. 
\end{proof}

	The next lemma shows that the condition that all finitely presented right $R$-modules are residually finite plays an analogous role to the Artin-Rees Lemma for $I$-adic completions of finitely generated modules over commutative Noetherian rings. 
	In particular, this lemma is classically used to show that if  $0\rightarrow A\rightarrow B\rightarrow C\rightarrow 0$ is a short exact sequence of finitely generated $R$-modules where $R$ is commutative and Noetherian and $I$ is some proper ideal of $R$, then the 
	 $I$-adic completion of $A$ is the same as the closure of $A$ in the $I$-adic topology of $B$.
	 This guarantees that the induced sequence $0\rightarrow \widehat{A}\rightarrow\widehat B\rightarrow \widehat C\rightarrow 0$ is exact. 
	 Moreover, it is also used to show that the $I$-adic completion of $R$ is flat as an $R$-module.

\begin{lem}\label{lem:fpMrf_implies_lerf}
	Let $R$ be a finitely generated ring such that all finitely presented right $R$-modules are residually finite. 
	Then, for every $d>0$ and every finitely generated submodule $L\leq R^d$, the profinite topology of $L$ coincides with the subspace topology from the profinite topology of $R^d$.
	
	In particular, given an exact sequence $0\to L\to R^d\to M\to 0$ of right $R$-modules with $L$ finitely generated,  the induced sequence
	$$0\to \widehat L\to \widehat R^d\to \widehat M\to 0$$ of right $\widehat R$-modules, is exact.
\end{lem}

\begin{proof}
	The second statement follows directly from the first one and Lemma \ref{indpro}, because the kernel of the map to $\widehat M$ is the closure of $L$ in $R^d$.

	To prove the first statement, let $L\leq R^d$ be a finitely generated submodule. 
	Then $R^d/L$ is finitely presented, so, by assumption, residually finite. 
	In other words, $L$ is the intersection of all finite index submodules of $R^d$ that contain $L$. 
	We need to show that for every $K\leq L$ of finite index, there is some $U\leq R^d$ of finite index such that $L\cap U\leq K$. 
	Since $R$ is a finitely generated ring, $K$ is a finitely generated submodule. 
	Therefore $R^d/K$ is residually finite and, as $L/K$ is finite, there is some $U\geq K$ of finite index in $R^d$ avoiding any set of non-trivial coset representatives of $K$ in $L$.  
	This means that $U\cap L=K$, as required. 
%		
%	
%	Let $M\leq R^d$ be a finitely generated submodule. 
%	Then $R^d/M$ is finitely presented, so, by assumption, residually finite. 
%	We need to show that $M=\bigcap\{U: U\in \mathcal{U}\}=:N$ where $\mathcal{U}$ is the set of all submodules of finite index (as subgroup) in  $R^d$ that contain $M$. 
%	Suppose that $x\in N\setminus M$. 
%	Since $R^d/M$ is residually finite, there exists some submodule $U/M$ of $R^d/M$ such that $U/M$ has finite index in $R^d/M$ and the image of $x$ in $R^d/M$ is not in $U/M$. 
%	But this means that $U\in \mathcal{U}$ and $x\notin U$, which is a contradiction as $N\leq U$.
%	
%	To show the second claim, we must prove that given any submodule $L\leq M$ of finite index there exists $U\leq R^d$ of finite index such that $M\cap U\leq L$. 
%	Since $R$ is finitely generated, $L$ is also finitely generated, so $R^d/L$ is residually finite. 
%	Since $M/L$ is finite, there is some $U\geq L$ of finite index in $R^d$ such that $M\cap U=L$. 
\end{proof}

%\begin{teo}\label{teo:Rhatflat_iff_Rcoherent}
%	Let $R$ be a residually finite countable ring such that all finitely presented $R$-modules are residually finite. 
%	The following are equivalent:
%	\begin{enumerate}
%		\item\label{it:Rhatflat+FPMRF} The profinite completion $\widehat{R}$ of $R$ is right flat. 
%		%	\item\label{it:Rhat/R_flat} The right $R$-module $\widehat{R}/R$ is flat. 
%		\item\label{it:Rcoherent+FPMRF} $R$ is left coherent.
%	\end{enumerate}
%\end{teo}

\begin{proof}[Proof of Theorem \ref{teo:Rhatflat_iff_Rcoherent}]
	\ref{it:Rhatflat+FPMRF} $\implies$ \ref{it:Rcoherent+FPMRF}.
	Let $L$ be a finitely generated right submodule of a free right module $R^k$. 
	 We want to show that $L$ is also finitely presented.
	  Consider the following exact sequence of right $R$-modules
	$$0\to L\to R^k\to M\to 0.$$
	Since $\widehat R$ is left flat and taking into account Lemma \ref{indpro} we obtain the commutative diagram
	
	$$\begin{array}{cccccccccc}
	0 &\to &  L \otimes_R \widehat{R} &  \rightarrow & R^k \otimes_R \widehat{R}&\rightarrow&  M \otimes_R \widehat{R}&\to 0\\
	& &{\Big\downarrow}{\scriptstyle\alpha_L}&& {\Big\downarrow}{\scriptstyle\alpha_{R^k}}&&{\Big\downarrow}{\scriptstyle\alpha_M}&\\
	& & \widehat L&\xrightarrow{\beta}  &\widehat R^k&\rightarrow& \widehat M&\to 0.
	\end{array}$$		
	By Lemma \ref{lem:alpha}, $\alpha_L$ is onto. Since $\alpha_{R^k}$ is bijective, the commutativity of the diagram implies that $\alpha_L$ should  be injective, and so, bijective.

	 Assume that $L$ is generated by $d$ elements. Therefore we have the following short exact sequence of right $R$-modules
	$$0\to I \to R^d \to L\to 0.$$
	
	Tensoring with $\widehat{R}$ and noting that $\widehat{R}$ is flat gives the short exact sequence
	$$0\to I\otimes_R \widehat{R} \rightarrow R^d\otimes_R \widehat{R} \rightarrow  L\otimes_R \widehat{R} \to 0.$$
	
	Let $\overline I$ be the closure of $I$ in $\widehat R^d$ and denote by $\overline {\alpha_I}:   I\otimes_R \widehat{R} \to \overline I$ the composition of $\alpha_I$ and the map $\widehat I\to \overline I$.
	 Then we also have, upon taking completions and applying Lemmas \ref{indpro} and \ref{lem:alpha} the commutative diagram
	$$\begin{array}{cccccccccc}
	0 &\to &  I\otimes_R \widehat{R} &  \rightarrow & R^d\otimes_R \widehat{R} &\rightarrow& L\otimes_R \widehat{R}&\to 0\\
	& &{\Big\downarrow}{\scriptstyle\overline{\alpha_I}}&& {\Big\downarrow}{\scriptstyle\alpha_{R^d}}&&{\Big\downarrow}{\scriptstyle\alpha_L}&\\
	0&\to & \overline I&\rightarrow  &\widehat R^d&\rightarrow & \widehat L&\to 0
	\end{array}$$		
	where  $\alpha_{R^d}$ and $\alpha_L$ are isomorphisms and $\gamma, \phi$ are continuous.  
	The commutativity of the diagram implies that $\overline{\alpha_I}$ is also an isomorphism.  %In particular, $\im \overline{\alpha_I}$ is closed.
	
	As $R$ is countable, $I$ is also countable.
	Enumerate the elements of $I=\{i_k\}_{k\in \N}$ and define, for each  $k\in \N$ the sets 
	$$I_k=\sum_{j=1}^ki_jR \text{ and } T_k= \sum_{j=1}^k i_j\otimes \widehat{R}.$$ 
	Thus,  $\overline I=\cup_{k\in \N}\overline{\alpha_I}(T_k)$ is a countable union of compact submodules. 
	Since  $\overline I$ is compact, by the Baire category theorem,  there exists $n\in \N$ such that $\overline{\alpha_I}(T_n)$ contains an open subset of  $\overline I$.
	 Hence for some $m\in \N$, $\overline{\alpha_I} (T_m)=\overline I$. 
	This implies  that the closure of $I_m=\overline{ \alpha_I}( I_m\otimes 1)$ in $R^k$ contains $I=\overline{\alpha_I}(I\otimes 1)$. 
	As $R^k/I_m$ is finitely presented, it is  residually finite,  by assumption on $R$.
	In other words, $I_m$ is closed in $R^k$, which yields $I=I_m$ and thus $L$ is finitely presented.

	 \ref{it:Rcoherent+FPMRF}$\implies $\ref{it:Rhatflat+FPMRF}.
	  We wish to show that $\Tor_1^R(M,\widehat{R})=0$ for every right $R$-module $M$. 
	Since every module is a direct limit of finitely presented modules, and $\Tor_1^R$ commutes with direct limits (in both variables), we may assume without loss of generality that $M$ is finitely presented.  
	Suppose that $M$ is defined by the short exact sequence 
	$$0\to L \to R^d \to M  \to 0$$
	where $L$ is finitely generated. 
	As $R$ is right coherent, $L$ is in fact finitely presented.
	Tensoring this exact sequence with $\widehat{R}$ we obtain the long exact sequence	
	$$0 \to \Tor_1^R(M,\widehat{R}) \to   L\otimes_R \widehat{R} \xrightarrow{i}  R^d \otimes_R \widehat{R}  \xrightarrow{p} M\otimes_R \widehat{R}\to 0. $$
	Then by Lemmas \ref{lem:alpha} and \ref{lem:fpMrf_implies_lerf}, we obtain the following commutative diagram with exact top and bottom rows	
	$$\begin{array}{cccccccccc}
	0 & \to \Tor_1^R(M,\widehat R)&\xrightarrow{\beta} & L\otimes_R \widehat{R} &\xrightarrow{i}&  R^d \otimes_R \widehat{R} &\xrightarrow{p} & M \otimes_R \widehat{R}&\to 0\\
	 & & &{\Big\downarrow}{\scriptstyle\alpha_L}&& {\Big\downarrow}{\scriptstyle\alpha_{R^d}}&&{\Big\downarrow}{\scriptstyle\alpha_M}&\\
	 & 0&\to & \widehat L&\xrightarrow{\eta}  &\widehat R^d&\xrightarrow{\pi} & \widehat M&\to 0,
	\end{array}$$
	where  $\alpha_L$,
	$\alpha_M$ and $\alpha_{R^d}$ are isomorphisms, because the right $R$-modules   $L$, $M$ and $R^d$ are finitely presented. 
	A diagram chase shows that   $\Tor_1^R(M,\widehat R)=0$, as required. 
 \end{proof}

 The condition on $R$ ``each finitely presented module  is residually finite" can also be encoded in a condition of flatness of some module. 

\begin{pro}\label{pro:profinitely_flat_iff_coherent_and_fprf}
	Let $R$ be a finitely generated  residually finite ring.   
	The following are equivalent:
	\begin{enumerate}
		\item the quotient module $\widehat{R}/R$ is left flat;
		\item every finitely presented right $R$-module is residually finite and  $\widehat{R}$ is left flat;
		\item\label{it:fprf+coherent} every finitely presented right $R$-module is residually finite and  $R$ is right coherent.
	\end{enumerate}
\end{pro}
\begin{proof}
	The last two items are equivalent by Theorem \ref{teo:Rhatflat_iff_Rcoherent}, so we show that the first two are equivalent. 
	Assume that  $\widehat R/R$ is left flat.
	 Since $R$ is also flat, $\widehat R$ is flat as well.
	 %extensions of flat modules are flat
	Let $M$ be a finitely presented right $R$-module. 
	Tensoring the short exact sequence $$0\to R\to \widehat R \to \widehat{R}/R \to 0$$ with $M$
	 yields the following commutative diagram with exact upper sequence:	
	$$\begin{array}{cccccccccc}
	\Tor_1^R(M,\widehat R) & \to &\Tor^R_1(M,\widehat R/R)&\to &M\otimes_R  R&\xrightarrow{\gamma} & M\otimes_R \widehat{R} &\to  & M\otimes_R (\widehat{R}/R) &\to 0\\ &&&
	& {\Big\downarrow}{\scriptstyle\sigma}&&{\Big\downarrow}{\scriptstyle\alpha_M}&& &\\
	&&&& M&\xrightarrow{\delta_M}  &\widehat M&  &  &
	\end{array}$$
	where $\gamma$ and $\sigma$    are the canonical maps.
	 By Lemma \ref{lem:alpha}, $\alpha_M$ is an isomorphism. 
	 It is obvious that the diagram commutes. 
	Since $\widehat{R}$ and $\widehat R/R$ are flat, $\Tor_1^R(M,\widehat{R})= \Tor^R_1(M,\widehat R/R)=0$.
	 Thus, $\delta_M$ is an embedding, and so $M$ is residually finite.
	
	Now assume that every finitely presented right $R$-module is residually finite and $\widehat R$ is left flat. 
	This means that in the previous diagram $\Tor_1^R(M,\widehat{R})=0$ and, as $\delta_M $ is an embedding, $\Tor_1^R(M,\widehat{R}/R)=0$ as well.	
\end{proof}

\begin{rem} Let $G$ be a topologically finitely generated profinite group and $R=\F_p[[G]]$. 
	By the result of Nikolov and Segal \cite{NS07}, the  map $\delta_R: R\to \widehat R$ is an isomorphism.
	So $\widehat{R}/R=0$ is flat, but obviously $R$ is not always coherent.
	Thus,  the  condition  that $R$ is finitely generated  is  not superficial. 
\end{rem}

Recall that $\F_p[[\widehat{G}]]\cong \widehat{\F_p[G]}$ for any group $G$. 
 We will see in  the  next  section that, if $G$  is a  virtually free group,  then $\F_p[[\widehat{G}]]$   is flat, for every prime $p$.
 Although we will not need it in the rest of the paper, it is interesting to note that the same is true for virtually polycylic groups:

\begin{pro}\label{pro:polycyclic_prof_flat}
	Let $G$ be a virtually polycyclic group and $p$ a prime number, then $\F_p[G]$ satisfies the conditions of Theorem \ref{teo:Rhatflat_iff_Rcoherent} and of Proposition \ref{pro:profinitely_flat_iff_coherent_and_fprf}. 
	In particular, $\F_p[[\widehat{G}]]$   is a  flat left $\F_p[G]$-module. 
\end{pro}
\begin{proof}
	This is a corollary of a result proved independently by Jategaonkar \cite{Jategaonkar} and Roseblade \cite{Roseblade} (based on a previous result of Roseblade \cite{Roseblade73}), answering a question of P. Hall; namely, every finitely generated abelian-by-polycylic-by-finite group is residually  finite (see  also \cite[7.2.1]{LennoxRobinson} for a proof).
	This implies that every finitely generated right $R$-module is residually finite, where $R=\F_p[G]$ for prime $p$ and $G$ is virtually polycyclic.
	It follows from \cite[Theorem 1]{PHall1954} that $k[G]$ is right Noetherian (and therefore right coherent) whenever $k$ is right Noetherian and $G$ is virtually polycyclic. 
	Since $R$ is finitely generated, because $G$ is, we conclude that $R$ enjoys the properties in Theorem \ref{teo:Rhatflat_iff_Rcoherent} and in Proposition \ref{pro:profinitely_flat_iff_coherent_and_fprf}(\ref{it:fprf+coherent}, so $\widehat{R}\cong \F_p[[\widehat{G}]]$ is a flat left $\F_p[G]$-module. 
\end{proof}

\section{Completed group algebras of completions of virtually free groups}\label{sec:limits_projective_modules}
%\todo{Which bits exactly are the ones that only hold for free groups? }

We now apply the results of the previous section, together with results from \cite{CoFR}, to show that the completed group algebra $\F_p[[\widehat{G}]]$ of a finitely generated virtually free group is a direct union of free $\F_p[G]$-modules. 
This will be a key ingredient in the next section, to show that if $\widehat{G}$ splits as a profinite free product, then so does $G$.

\begin{definition}
A \textbf{free left ideal ring}, or\textbf{ left fir} for short, is a ring in which all left ideals are free left $R$-modules of unique rank.
 \textbf{Right firs} are defined correspondingly and a \textbf{fir} is a left and right fir. 
\end{definition}
%By \cite[Corollary 1.2]{CoFR}, a left fir  has invariant basis number.
 %In particular, it is always an integral domain.
 Note that if $R$ is a fir then %it is both left and right \textbf{coherent} -- 
 every finitely generated right (or left) ideal is a finitely presented right (or left) $R$-module. 
 
 \begin{pro}[\cite{CoFR} Theorem A.6 and Corollary A.7, p. 554]\label{prop:coherent_equivalent}
 For a ring $R$, the following are equivalent:
 \begin{enumerate}
 \item  $R$ is right coherent;
 \item every finitely generated right ideal of $R$ is a finitely presented right $R$-module.
% \item  the direct product of any family of flat right $R$-modules is flat;
% \item every direct power $R^I$ is right flat;
% \item  the left annihilator of any matrix over $R$ is finitely generated, i.e. given $A\in \Mat(n\times m, R)$ if $BA=0$ for some $B\in \Mat(I\times n, R)$ there exist $C\in \Mat(r\times n, R), D\in \Mat(I\times r, R)$ such that $B=DC$ and $CA=0$.
\end{enumerate}
 
 Furthermore, over a right coherent ring, the intersection of any pair of finitely generated submodules of a free right module is finitely generated.
\end{pro}

%\begin{definition}
%A finitely generated residually finite group $G$ is called {\bf $p$-profinitely flat} if $\F_p[G]$ is right profinitely flat; it is  {\bf profinitely flat} if $\Z[G]$ is right profinitely flat. 
%\end{definition}

% Notice that $\widehat{\F_p[G]}\cong \F_p[[\widehat G]]$ and $\widehat{\Z[G]}\cong \widehat {\Z}[[\widehat G]]$. We leave the study of these new notions for future research. 
% In this paper we will only need the following result.

\begin{pro}
	Let $F$ be a finitely generated free group and $p$ a prime.
	Then $\F_p[[\widehat{F}]]$ is a flat left $\F_p[F]$-module.
\end{pro}
\begin{proof}
%First observe that $\widehat{\F_p[F]}\cong \F_p[[\widehat F]]$. 	
 The ring $\F_p[F]$ is a fir (\cite{Be74}).
 Thus, it is also coherent.
 We claim that every finitely presented $\F_p[F]$-module is residually finite. 
 If this holds, then, observing that  $\widehat{\F_p[F]}\cong \F_p[[\widehat F]]$, Theorem \ref{teo:Rhatflat_iff_Rcoherent} implies that $\F_p[[\widehat{F}]]$ is a flat left $\F_p[F]$-module. 
 
 To show the claim, write $R:=\F_p[F]$, let $M\cong R^n/I$ be a finitely presented right $R$-module, with $n\in \N$ and $I$  finitely generated. 
 We wish to show that given any $z\in R^n\setminus I$ there is a right submodule $L\supseteq I$ such that $z\notin L$ and $R^n/L$ is finite. 
 We proceed by induction on $n$. 
 If $n=1$, then \cite[Theorem 4.3]{RR94} implies that for any $z\in R \setminus I$ there is a finitely generated ideal $L\supseteq I$ such that $z\notin L$ and $R/L$ is finite.

 Suppose the claim true for $n-1\geq 1$ and let $z\in R^n \setminus I$. 
 Denote by $R_1$ the first summand in $R^n$.
 If $z\notin R_1+I$, we consider the quotient $R^n/(R_1+I)\cong R^{n-1}/I_1$ where $R^{n-1}=R^n/R_1$ and $I_1=R_1+I/R_1$. 
 Since $R_1$ and $I$ are finitely generated right $R$-modules,  $I_1$ is a finitely generated right submodule of $R^{n-1}$ that does not contain the image $z_1$ of $z$ in $R^{n-1}$. 
 By inductive hypothesis, there is a finitely generated right submodule $L_1\supseteq I_1$ of $R^{n-1}$ such that $z_1\notin L_1$ and $R^{n-1}/L_1$ is finite. 
 The preimage of $L_1$ in $R^n$ is the desired right submodule of $R^n$. 
 If $z\in R_1+I$, consider its image $z_1$ under the isomorphism $R_1+I/I\cong R_1/R_1\cap I$, so $z_1\notin R_1\cap I$. 
 Since $\F_p[F]$ is coherent, Proposition \ref{prop:coherent_equivalent} guarantees that $R_1\cap I$ is a finitely generated right ideal of $R_1$. 
 Theorem 4.3 of \cite{RR94} then gives a finitely generated right ideal $L_1\supseteq R_1\cap I$ of $R_1$ not containing $z_1$ and such that $R_1/L_1$ is finite. 
 The preimage of $L_1$ in $R_1+I$ is our desired $L$. 
 %
% The claim then follows by induction.	
\end{proof}

%\begin{definition}
%	Let $R$ be a ring and $M$ a right $R$-module. Say that $M$ is \textbf{semifree} if every finitely generated submodule is free. 
%\end{definition}

\begin{pro}[See Proposition 1.4.5 in \cite{CoFR}]\label{prop:flat iff all fg submodules free}
	Let $R$ be a fir and $M$ a left $R$-module.
	Then $M$ is flat if and only if every finitely generated left submodule of $M$ is free.
\end{pro}

Putting together the last two results yields the following:

\begin{cor}\label{locfree}
	Let $F$ be a finitely generated  free group.
	Then every finitely generated submodule of the left $\F_p[F]$-module $\F_p[[\widehat F]]$ is free.
\end{cor}

%\begin{proof}
%	Since  $\F_p[F]$ is a fir (\cite{Be74}),  any flat $\F_p[F]$-module satisfies the condition in the statement (\cite[Proposition 1.4.5]{CoFR}. The claim follows as $\F_p[[\widehat F]]$ is flat.	
%\end{proof}

\begin{cor}\label{locvfree}
	Let $G$ be a finitely generated virtually free group.
	 Then the left $\F_p[G]$-module $\F_p[[\widehat G]]$ is isomorphic to a direct union of free left $\F_p[G]$-modules.
\end{cor}

\begin{proof}
%	If  $\F_p[[\widehat G]]/\F_p[G]$ is a direct union of free $\F_p[G]$-modules, then clearly $\F_p[[\widehat G]]$ is also a direct union of free $\F_p[G]$-modules. Thus, it is enough to prove the proposition for  $\F_p[[\widehat G]]/\F_p[G]$.

Let $F$ be a free subgroup of $G$ of finite index and denote by  $\overline F$ the closure of $F$ in $\widehat G$.
Then $\overline F\cong \widehat F$ because $F$ is of finite index in $G$.
Now, $\F_p[[\widehat G]]\cong  \F_p[G]\otimes_{\F_p[F]}\F_p[[\overline F]]$. 
By Corollary \ref{locfree},  $\F_p[[\widehat F]]=\F_p[[\bar F]]$ is isomorphic to a direct union of free $\F_p[F]$-modules, therefore, as tensor products commute with direct unions,  $\F_p[[\widehat G]]\cong  \F_p[G]\otimes_{\F_p[F]}\F_p[[\overline F]]$ is isomorphic to a direct union of free $\F_p[G]$-modules.
\end{proof}

\begin{cor} \label{cor:completedgroupalg=directunionprojectives}
	Let $G$ be a finitely generated virtually free group. 
	Let $n$ be a product of a finite set of different primes and $k=\Z/n\Z$.
	Then the $k[G]$-module  $k[[\widehat G]]$ is isomorphic to a direct union of projective $k[G]$-modules.
\end{cor}

\section{Free factors of virtually free groups}\label{sec:Dicksproof}

This section is devoted to the proof of the equivalence of items a) and c) in Theorem \ref{teo:vfree}. 
Recall that this equivalence is the following:

\begin{teo}\label{teo:profinite_freefactor_iff_abstract_freefactor}
Let $G$ be a finitely generated group that is virtually free and let $H\leq G$ be finitely generated. Denote by $\overline{H}$ the closure of $H$ in the profinite completion $\widehat{G}$ of $G$. Then  $\widehat{G}$ splits as a free profinite product $\widehat{G}=\overline{H}\coprod K$ if and only if $G=H\ast L$ splits as an abstract free product. 
\end{teo}

The ``if" part follows from the facts that ``profinite completion commutes with free product"; i.e., $\widehat{H}\coprod \widehat{L}=\widehat{H\ast L}$ and that the profinite topology of $H\ast L$ induces the full profinite topologies on $H$ and $L$ (see \cite[Corollary 3.1.6]{RibesZalesskii}).

The rest of this section is devoted to the ``only if" implication. 
It relies on a decomposition theorem due to  Dicks and Dunwoody (\cite{Di81,Dicks_book}).
To set notation, we recall the relevant Bass--Serre theory, mostly following \cite{Dicks_book}. 

\subsection{Prerequisites from Bass--Serre Theory}

A \textbf{graph} $Y=V(Y) \sqcup E(Y)$ is the disjoint union of a set $V=V(Y)$ of \textbf{vertices} and a set $E=E(Y)$ of \textbf{edges}, given with two maps  $i,t:E\rightarrow V$ that associate to each edge its \textbf{initial} $i(e)$, respectively, \textbf{terminal} $t(e)$ vertex. 
Loops ($i(e)=t(e)$) are allowed. 
In other words, we consider directed multigraphs. 
$Y$ is \textbf{connected} if it is connected as an undirected graph. 
A \textbf{tree} is a connected graph which has no undirected cycles.
By Zorn's Lemma, every connected graph has a maximal subtree. 
For our purposes, $Y$ will always be finite. 

A \textbf{graph of groups} $(\G,Y)$ consists of a graph $Y$, vertex groups $\{\G(v) : v\in V(Y)\}$ and edge groups $\{\G(e) : e\in E(Y)\}$ and, for each $e\in E(Y)$, two group morphisms $i_e, t_e: \G(e)\rightarrow \G(i(e)), \G(t(e))$. 
If all the morphisms $i_e$,$t_e$, $e\in E(Y)$ are injective, the graph of groups is called \textbf{faithful.} 

Given a connected graph of groups $(\G,Y)$ and a maximal subtree $T$ of $Y$, the \textbf{fundamental group} $G=\pi(\G,Y,T)$ of $(\G,Y)$ with respect to $T$ is the group with presentation:
 \begin{itemize}
 	\item generators $\{\G(v), q_e: v\in V(Y), e\in E(Y)\}$;
 	\item relations of each $\G(v), v\in V(Y)$;
 	\item relations $(i_e(g))^{q_e}=t_e(g)$ for each $g\in \G(e), e\in E(Y)$;
 	\item relations $q_e=1$ if $e\in E(T)$ 	
\end{itemize}

One of the first results of Bass--Serre theory (see \cite[II.2.4]{Dicks_book}, \cite[I.5.20]{Serre_trees}) is that any other maximal subtree yields an isomorphic fundamental group. 
Another result (\cite[II.2.7]{Dicks_book}) states that if $(\G,Y)$ is connected and faithful, then the canonical maps  $\G(v)\rightarrow G, v\in V(Y)$ are injective  and we may identify each $\G(v)$ with its image.

\begin{rem}\label{rem:edge_groups_are_not_vertex_groups}
	Let $(\G,Y)$ be a connected faithful graph of groups with maximal subtree $T$. 
	Suppose there is some edge $e$ in $T$ with $i(e)=v, t(e)=w$ and such that one of $i_e$ or $t_e$ is an isomorphism. 
	Without loss of generality, assume that $i_e$ is an isomorphism, so that we may identify $\G_e=\G_v$. 
	Then, in the fundamental group $\pi(\G,Y,T)$ we have that $(\G_v)^{q_e}=\G_v=t_e(\G_v)\leq \G_w$. 
	
	Consider the graph $Y'$ and maximal tree $T'$ obtained from $Y$ by identifying $v$ and $w$ along $e$ (leaving the other possible edges between $v$ and $w$ as loops). 
	Then $\pi(\G,Y,T)=\pi(\G,Y',T')$ because $\G_v\leq \G_w$ in $\pi(\G,Y,T)$. 
	
	Doing this procedure for all edges $e$ of $T$ as above, we may assume that if $G=\pi(\G,Y,T)$, then for every edge $e$ of $T$, neither $i_e$ nor $t_e$ are isomorphisms. 
	
	Since choosing a different maximal subtree does not change the isomorphism type of the group, we may further assume that $G$ is the fundamental group of a connected faithful graph of groups such that for every edge $e$ joining different vertices neither $i_e$ nor $t_e$ are isomorphisms. 
	%
	%
	%We could still have a Baumslag-Solitar type of situtation where the loop in the graph of groups is never included in a maximal subtree and induces an isomorphism of the vertex group with a subgroup of itself.  
\end{rem}

Let $G$ act on a tree $X$ and let $S\subset X$ be a \textbf{connected transversal} for the $G$-action.
This is a subset of $X$ (in general, not a subgraph%, as the terminal map is not defined for all edges
) that is in bijection with $G\backslash X$, any two of whose vertices are connected by a path in $S$ and such that $i(e)\in S$ for every $e\in S$.  
For each $e\in E(S)$, there is a unique vertex of $S$ in the same orbit as $t(e)$, so we choose $q_e\in G$ such that $q_e\cdot t(e) \in S$, taking $q_e=1$ if $t(e)\in S$.
The collection $(q_e : e\in E(S))$ is a \textbf{connecting family} for $S$. 
For each $s\in S$, denote by $G_s$ its stabiliser in $G$. 
Note that $G_{q_e\cdot t(e)}=q_eG_{t(e)}q_e^{-1}$ and that for every $e\in E(S)$, $G_e=G_{i(e)}\cap G_{t(e)}=G_{i(e)}\cap q_e^{-1}G_uq_e$ where $u=q_e\cdot t(e)$, so that $G_e^{q_e}\leq G_{q_e\cdot t(e)}$. 
These are exactly the conjugation relations that appear in the definition of the fundamental group of a graph of groups. 

Let $(\G,Y)$ be a connected graph of groups with maximal subtree $T$. 
Write $\pi=\pi(\G,Y,T)$. 
The \textbf{standard tree} of $(\G,Y,T)$ is the graph $\tilde{X}=\tilde{X}(\G,Y,T)$ with, respectively, vertex and edge sets $$V(\tilde{X})=\bigsqcup_{v\in V(Y)}\pi/\G(v), \qquad E(\tilde{X})=\bigsqcup_{e\in E(Y)}\pi/G_{e}$$
and incidence maps given by $$i(g\G(e))=g\G(i(e)), \quad t(g\G(e))=gq_e\G(t(e)).$$ 
That $\tilde{X}$ is indeed a tree is a key result of Bass--Serre theory (\cite[I.5.3]{Dicks_book}, \cite[I.5.12]{Serre_trees}).

\begin{teo}[Structure Theorem of Bass--Serre theory]
 Let $G$ act on a connected graph $X$ with connected transversal $ S\subset X$, maximal subtree $T\subset S$ and connecting family $(q_e\in G: e\in E(S))$. 
 Then there is a faithful connected graph of groups $(\G, G\backslash X)$ defined by $\G(s)=G_s$ for $s\in S$ and $i_e, t_e: G_e\rightarrow G_{i(e)}, G_{(q_e\cdot t(e))}$ given, respectively,  by $g \mapsto g, g^{q_e}$, for each $e\in E(S)$. 
 
 There is a surjective homomorphism $\varphi:\pi=\pi(\G,G\backslash X, T) \rightarrow G$, uniquely determined by the inclusion maps $\G(v)\rightarrow G$ and $q_e\rightarrow q_e$, $v\in V(Y), e\in E(S)$. 
 
 There is an equivariant universal covering map from the standard tree $\tilde{X}(\G,G\backslash X, T)$ to $X$, given by $g\G(s)\mapsto \varphi(g)\cdot s, g\in\pi, s\in S$.
 
 The map $\tilde{X}\rightarrow X$ is bijective if and only if $\varphi$ is an isomorphism. 
\end{teo}

%\noindent\textbf{Notation.}
 Let $G=\pi(\G,Y,T)$ be the fundamental group of a connected faithful graph of groups, with respect to a maximal subtree $T$. 
Since we can identify each $\G(v), v\in V(Y)$ with its image in $G$, we write $G_y:=\G(y)$ for each $y\in Y$. 
For every $e\in E(Y)$, write $i(G_e)\leq G_{i(e)}$ and $t(G_e)\leq G_{t(e)}$ for the respective images of $G_e$ under $i_e, t_e$. 
Conjugating by $q_e$ induces an isomorphism $i(G_e)\rightarrow t(G_e)$, which is the identity if $e\in T$, in which case we identify $G_e$ with $i(G_e)=t(G_e)$.
\begin{pro}[Proposition II.3.1 of \cite{Dicks_book}]\label{pro:intersections_of_vertex_groups}
	With the above notation, $\G_{i(e)}\cap \G_{t(e)}^{q_e}=i_e(\G(e))$ for any $e\in E(Y)$.	
	Given any $v, w\in V(Y)$, let $P$ be the set of edges of $T$ in the geodesic joining $v, w$. Then $G_v\cap G_w=\bigcap_{e\in P}G_e$.  
\end{pro}
	
\begin{rem}\label{rem:suff_cond_free_factor}
Suppose that $G=\pi(\G,Y,T)$ is the fundamental group of some connected faithful graph of groups and that $H\leq G$ is one of the vertex groups. 
$H$ is a free factor of $G$ if all edge groups incident with $H$ have trivial image in $H$. 
By Proposition \ref{pro:intersections_of_vertex_groups}, 
this holds if  $H\cap G_v^g$ is trivial for every vertex group distinct from $H$ and every $g\in G$ and if $H\cap H^g$ is trivial for every $g\in G\setminus H$. 
\end{rem}

The last result that we need is the crucial link between derivations and decompositions into fundamental groups of graphs of groups. 
A derivation $d: G\rightarrow P$ is \textbf{inner} if there is some $p\in P$ such that $d(g)=(1-g)p$ for all $g\in G$. 

\begin{teo}[Theorem III.4.6 of \cite{Dicks_book}]\label{teo:dicks_derivation_to_decomposition}
	Let $G$ be a group and $R$ a ring. 
	Suppose that $P$ is a projective $R[G]$-module, and $d:G \rightarrow P$ is a derivation such that $G$ is generated by $\ker d$ together with finitely many elements. 
	Then $G$ decomposes as the fundamental group of a faithful connected finite graph of groups with finite edge groups, $\ker d$ as one of the vertex groups, and such that the derivation $d$ is inner when restricted to each vertex group. 
\end{teo}
Note that if $d$ is inner when restricted to vertex groups, then $d$ is also inner on conjugates of vertex groups: for every $f\in G$ and $g\in G_v$ we have 
$d(fgf^{-1})=(1-fgf^{-1})d(f)+fd(g)=(1-fgf^{-1})(d(f)+fp).$

\subsection{Proof of Theorem \ref{teo:profinite_freefactor_iff_abstract_freefactor}, continued}
Suppose that $G$ is finitely generated, virtually free and $H\leq G$ satisfies that $\widehat{G}=\overline{H}\coprod K$ for some closed subgroup $K$ of $\widehat{G}$.

Let $U$ be a free, normal subgroup of finite index in $G$ and let $k=\Z/n\Z$ where $n$ is the product of all prime factors of $|G:U|$. 
Lemma \ref{lem:freeprofprod_implies_idealsdecompose} yields a continuous homomorphism $$\phi:I_{k[[\widehat G]]}\rightarrow I_{k[[ K]]}^{\widehat{G}}$$ of $k[[\widehat{G}]]$-modules, which is simply projection to the second component. 
Note that these are also $k[G]$-modules. 
Since $G$ is finitely generated, the image of $\{1-g:g\in G\}\subseteq I_{k[[\widehat G]]}$ under $\phi$ is a finitely generated $k[G]$-module. 
By Corollary \ref{cor:completedgroupalg=directunionprojectives}, this finitely generated $k[G]$-module lies in a direct union of projective $k[G]$-modules (which we can take to be finitely generated, since every module is the direct union of its finitely generated submodules). 
Let $P$ be a projective finitely generated $k[G]$-module containing $\phi(I_{k[ G]})$ and denote  by $\varphi:I_{k[ G]}\rightarrow P$ the restriction of $\phi$ to $I_{k[ G]}$, with image in $P$. 

Then $\varphi$ induces a derivation $$d:G\rightarrow P, \ g\mapsto \varphi(1-g).$$ 
Notice that $\varphi(1-g)=0$ exactly when $g\in \overline{H}\cap G$. 
Since $G$ is virtually free,  all finitely generated subgroups are closed in the profinite topology (that is, $G$ is  LERF), so $\overline{H}\cap G=H$. 
We have thus obtained a derivation $d:G\rightarrow P$ to a projective $k[G]$-module whose kernel is exactly $H$. 
Theorem \ref{teo:dicks_derivation_to_decomposition} then implies that $G$ decomposes as the fundamental group of a faithful connected finite graph of groups $Y=(V,E)$, with finite edge groups, with $H=G_u$ for some $u\in V$ and such that the derivation $d$ is inner when restricted to any conjugate of any vertex group.

Note that by Remark \ref{rem:edge_groups_are_not_vertex_groups}, we can assume that for every $e\in E$ with distinct endpoints $v,w$, neither of the injections $G_e\rightarrow G_v,G_w$ is surjective. 
Following Remark \ref{rem:suff_cond_free_factor}, we will show that $H\cap G_v^g=1$ for every $g\in G$ and $H\cap H^g=1$ for every $g\in G\setminus H$.

% \color{blue} We can modify the graph $Y$ and assume that there is no a vertex $v_0\ne v\in V(Y)$ connected to  $v_0$ by  an edge $e$ and such that $G_e\to G_v$ is an isomorphism. \color{black}
%We will show that  \color{blue}  $G_e=1$ if $e$ is an edge having $v_0$ as one of its verteces\color{black}, which implies that $H$ is a free factor of $G$. %$G=\langle H, (G_v)_v, (q_e)_{e\in E(Y)} \mid relations of H, G_v, finite edge stabs, conjugation by q_e's \rangle= H \ast \langle (G_v)_v, (q_e)\rangle$.

\

\
\noindent\textbf{Case 1: $v\neq u$ and $G_v$ is infinite.}
First we reduce to the case where $H$ is finite. 
Recall that $U$ was taken to be a free, normal subgroup of finite index in $G$. 
Then $H':=H/(H\cap U)$ is finite and, since $H\cap U$ is torsion-free, it intersects trivially all edge groups, so that $G':=G/\langle H\cap U\rangle^G$ has the same presentation as $G$, with the only addition that $H$ is replaced by the finite group $H'$.
In particular, each vertex group $G_v$ distinct from $H$ maps isomorphically to a subgroup of $G'$, which, abusing notation, we shall also call $G_v$.

We then have  $\widehat{G'}=\overline{H'}\coprod K=H'\coprod K$ and so $I_{k[[\widehat{G'}]]}=I_{k[[H']]}^{\widehat{G'}}\oplus I_{k[[K]]}^{\widehat{G'}}$.
This gives the following commutative diagram:
$$\begin{array}[c]{ccccccc}
G & \hookrightarrow & \widehat{G} &\hookrightarrow & I_{k[[\widehat{G}]]} & \stackrel{\phi}{\rightarrow} & I_{k[[K]]}^{\widehat{G}}\\
\downarrow && \downarrow && \downarrow && \downarrow	\\
G' & \hookrightarrow & \widehat{G'} & \hookrightarrow & I_{k[[\widehat{G'}]]} & \stackrel{\phi'}{\rightarrow} & I_{k[[K]]}^{\widehat{G}}
\end{array}$$
where the downward arrows are the quotient maps by the appropriate image of $\langle H\cap U\rangle ^G$. 

Note that $d$ is the composition of all maps on the top row of the diagram. 
Denote by $d'$ the composition of the maps on the bottom row. 
This is in particular a derivation on $G'$, whose kernel is $H'$, and that is inner on every conjugate of  $G_v\leq G'$ for every $v\in V$.

Fix now some $v\neq u$ such that $G_v$ is infinite and denote by $L$ any conjugate of $G_v$. 
Because $d'%:G'\rightarrow I_{k[[\widehat {G^\prime}]]}
$ is inner on $L$, there is some $m\in I_{k[[\widehat {G^\prime}]]} $ such that $d'(g)=(1-g)m$ for every $g\in L$. 
Recalling that $1-g=d'(g)+\left(\text{the }I_{k[H]}^{\widehat{G^\prime}}\text{-part of }1-g\right)$, we have that $1-g \equiv(1-g)m$ modulo $I_{k[H]}^{\widehat{G^\prime}}$ for every $g\in L$. 
Equivalently,  $(1-g)(1-m)\equiv 0$ modulo $I_{k[H^\prime]}^{\widehat{G^\prime}}$ for every $g\in L$, so the image of $1-m$ in $k[[\widehat{G'}/H']]$ is a fixed point of $L$. 
 
Any fixed point of $L$ in $k[[\widehat{G'}/H']]$ is also a fixed point of the closure $\overline{L}$ in $\widehat{G'}$, because $L$ is dense in $\overline{L}$ (as $G'$ is residually finite) and the action of $\widehat{G'}$ on $k[[\widehat{G'}/H']]$ is continuous.

Now,  $\overline{H'}=H'$ is finite, so for every $N\unlhd_o\widehat{G'}$ and $g\in \widehat{G'}$ we have 
$$|\overline{L}:\overline{L}\cap N|=|\overline{L}:\overline{L}\cap (H')^gN|\cdot |\overline{L}\cap (H')^gN: N|\leq |\overline{L}:\overline{L}\cap (H')^gN| \cdot |H'|.$$
Therefore $p^\infty$ divides $\lcm_{N\unlhd_o\widehat{G'}} |\overline{L}:\overline{L}\cap (H')^gN|$ if and only if it divides $\lcm_{N\unlhd_o\widehat{G'}}|\overline{L}:\overline{L}\cap N|=|\overline{L}|$. 

Because $G'$ is LERF, the profinite completion of $L$ coincides with its closure $\overline{L}$ in $\widehat{G'}$. 
As $L$ is virtually torsion-free,   $\overline{L}$ contains a copy of $\widehat{\Z}$; in particular, $p^\infty$  divides $|\overline{L}|$.
Now Lemma \ref{lem:no_fixed_points_modH} implies that $\overline{L}$ has no non-trivial fixed points when it acts on $k[[\widehat{G'}/H']]$. 
In particular, $(1-m)$ is congruent to 0 in $k[[\widehat{G'}/H']]$, which is equivalent to $1-m \in I_{k[H^\prime]}^{\widehat{G^\prime}}$.

Recall that the kernel of $d'$ is $H'$, so $d'(h)=(1-h)m=0$, that is, $hm=m$, for all $h\in L\cap H'$.

Let $G'/N$ be a finite quotient of $G'$ in which $H'$ injects. The image $\widetilde{m}\in k[G'/N]$ of $m$ is a finite sum $\widetilde{m}=\sum_{\gamma\in G'/N}r_{\gamma}\gamma$. 
Since $m\equiv 1$ in $k[[\widehat{G'}/H']]$, the coefficients of $\gamma\in H'N/N\cong H'$ in $\widetilde{m}$ must sum to 1: $\sum_{\gamma\in H'N/N}r_{\gamma}=1$. 
Moreover, because $\widetilde{m}$ is invariant under $H'\cap L$, we must have that $r_{\gamma}=r_{h\gamma}$ for all $h\in H'\cap L$ and $\gamma\in G'/N$. 
In particular, $$1=\sum_{\gamma\in H'N/N}r_{\gamma}=\sum_{\delta\in H'N/(H'\cap L)N}\left(r_{\delta}|H'\cap L|\right),$$
so $|H'\cap L|$ is invertible in $k$. 

By the choice of $k$, the invertible elements are those that are coprime with all primes that divide $|G:U|$. 
But $|H'\cap L|$ divides $|H'|=|HU:U|$, which divides $|G:U|$. 
So $H'\cap L=1$. 

This means that the image of $H\cap L\leq G$ in $G'$ is trivial, which implies that $H\cap L\leq H\cap U$.
 But $U$ is torsion-free, so in fact $H\cap L=1$. 
 
 \

\noindent\textbf{Case 2: $v\ne u$ and $G_v$ is finite.}
%
%If $G_v$ is finite, then it is a closed subgroup of $\widehat{G}$. 
Like in the previous case, denote by $L$ any conjugate of $G_v$ in $G$. 
By \cite[Theorem 9.5.1]{RibesZalesskii}, since $L$ is finite, it is conjugate in $\widehat{G}$ to a subgroup of $\overline{H}$ or $K$. 
If the latter, we are done, as any conjugate of $K$ intersects $\overline{H}$ trivially, by \cite[Corollary 3.1.2]{MelnikovZalesskii}. 
The same result guarantees that a conjugate of $\overline{H}$ distinct from $\overline{H}$ has trivial intersection with $\overline{H}$ (see also \cite[Theorem 9.1.12]{RibesZalesskii}). 
So we are left with the case $L\leq \overline{H}$, which means that $L\leq H$
because $H$ is closed in the profinite topology of $G$.

Suppose that $L=G_v^g$, so that $L$ is the stabiliser of $gv$ in the standard tree.
Let $gw$ be the next vertex in the path between $gv$ and $u$ in the standard tree. 
Since $G_v^g\leq H=G_u$, we also have that $G_v^g\leq G_w^g$ and therefore $G_v\leq G_w$, where inclusion is induced by the edge homomorphisms of the edge $e\in E$ between $v,w\in V$. 
This contradicts the assumption on $G$ that we made following Remark \ref{rem:edge_groups_are_not_vertex_groups}. 

% \color{blue} But this is impossible by our assumption  that there is no a vertex $v_0\ne v\in V(Y)$ connected to  $v_0$ by the edge $e$ and such that $G_e\to G_{v}$ is an isomorphism. \color{black}

\

\

\noindent\textbf{Case 3: $v=u$.}
Here we must show that $H^g\cap H=1$ for every $g\notin H$. 
No such $g$ is contained in $\overline{H}$, as $H=G\cap \overline{H}$. 
By \cite[Theorem 9.1.12]{RibesZalesskii}, this implies that $\overline{H}^g\cap \overline{H}=1$, which finishes the argument.\hfill \qedsymbol


\begin{thebibliography}{0}  
	%\bibitem{Ba68} G. Baumslag, On the residual nilpotence of certain one-relator groups.
	%Comm. Pure Appl. Math. 21 (1968), 491--506.
	
	%\bibitem{Ba05}  G. Baumslag,  Parafree groups. Infinite groups: geometric, combinatorial and dynamical aspects, 1--14, Progr. Math., 248, Birkhäuser, Basel, 2005.
	
	\bibitem{Baumslag} G. Baumslag, Residually finite groups with the same finite images, Compos. Math., Volume 29 (1974), pp. 249-252.
	
	\bibitem{Be74} G. M. Bergman,   Coproducts and some universal ring constructions. Trans. Amer. Math. Soc. 200 (1974), 33--88.
	
	
	\bibitem{Zalesskii_vfree} V. R. de Bessa, A. L. P. Porto, P. A. Zalesskii, The profinite completion of accessible groups. Monatsh Math (2022). \url{https://doi.org/10.1007/s00605-022-01789-9}	
		
	
	\bibitem{BMRS18} M. R. Bridson, D. B. McReynolds, A. W. Reid and R. Spitler, Absolute profinite rigidity and hyperbolic geometry. Ann. Math., Volume 192 (2020) no. 3, p79--719.
	
	
	\bibitem{BMRS20} M. R. Bridson, D. B. McReynolds, A. W. Reid and R. Spitler,  On the profinite rigidity of triangle groups. Bull. London Math. Soc., 53: 1849-1862.
	

	%\bibitem{Ch76}  I. M. Chiswell,  Exact sequences associated with a graph of groups. J. Pure Appl. Algebra 8 (1976),    63--74.
	
	%\bibitem{Co72} D. E. Cohen,  Groups of cohomological dimension one.
	%Lecture Notes in Mathematics, Vol. 245. Springer-Verlag, Berlin-New York, 1972.
	
	\bibitem{CoFR} P. M. Cohn, Free rings and their relations. Second edition. London Mathematical Society Monographs, 19. Academic Press, Inc. [Harcourt Brace Jovanovich, Publishers], London, 1985.
	
	
	\bibitem{Dicks_book}
	W.~Dicks, \textit{Groups, Trees and Projective Modules}. Lecture Notes in Mathematics. Vol. 790. Springer-Verlag Berlin Heidelberg New York, 1980. 
	
	\bibitem{Di81} \rule[0.48ex]{3em}{0.14ex}\space, On splitting augmentation ideals. Proc. Amer. Math. Soc. 83 (1981), no. 2, 221--227.
	
	
	
	%\bibitem{FJ08}  M. D. Fried, M.  Jarden, M. Field arithmetic. Third edition. Revised by Jarden. Ergebnisse der Mathematik und ihrer Grenzgebiete. 3. Folge. A Series of Modern Surveys in Mathematics [Results in Mathematics and Related Areas. 3rd Series. A Series of Modern Surveys in Mathematics], 11. Springer-Verlag, Berlin, 2008.
	
	%  \bibitem{Gra19} J. Gr\"ater,   Free Division Rings of Fractions of Crossed Products of Groups With Conradian Left-Orders,  	arXiv:1910.07021 (2019).
	
	\bibitem{GrunewaldScharlau}F. J. Grunewald and R. Scharlau,
	A note on finitely generated torsion-free nilpotent groups of class 2, J.~ Algebra, 58 (1979),  162--175.
	
	\bibitem{PHall1954} P.~ Hall, Finiteness conditions for soluble groups. Proc. London Math. Soc. (3) 4 (1954), 419--436.
	
	%\bibitem{Ha59}  P. Hall, On the finiteness of certain soluble groups. Proc. London Math. Soc. (3) 9 (1959), 595--622.
	
	%\bibitem{HMP20}  L. Hanany, C. Meiri,  D.  Puder,  Some orbits of free words that are determined by measures on finite groups. J. Algebra 555 (2020), 305--324. 
	
	\bibitem{Ha11} B. Hassan,  Coherent rings and completion. Master's thesis, University of Glasgow (2011, \url{http://theses.gla.ac.uk/2514/1/2011HassanBMScR.pdf}.
	
	%\bibitem{Ja19un} A. Jaikin-Zapirain, The universality of Hughes-free division rings, preprint (2019), \url{http://matematicas.uam.es/~andrei.jaikin/preprints/universal.pdf}.
	
	%  \bibitem{Ja20ba} A. Jaikin-Zapirain, Free  $\Q$-groups are residually torsion-free nilpotent, preprint (2020), \url{http://matematicas.uam.es/~andrei.jaikin/preprints/baumslag.pdf}.
	
	\bibitem{Jategaonkar} A.~V.~ Jategaonkar, Integral group rings	of polycyclic-by-finite groups.  J.~ Pure Appl.~ Algebra, 4, (1974), 337--343. 
	
	 \bibitem{Kourovka}  E.I. Khukhro  and V.D. Mazurov (eds.). The Kourovka notebook. Unsolved problems in group theory, 18th ed., Ross. Akad. Sci. Sib. Div., Inst. Math., Novosibirsk 2014, 227 pp.
	
	
	%\bibitem{Ki20}  D. Kielak,  Residually finite rationally solvable groups and virtual fibring. J. Amer. Math. Soc. 33 (2020),   451--486.
	
	
	%\bibitem{LW54}  S. Lang, A. Weil,   Number of points of varieties in finite fields. Amer. J. Math. 76 (1954), 819--827.
	
	\bibitem{LennoxRobinson} J.~C.~ Lennox and D.~J.~S.~ Robinson, The theory of infinite soluble groups. Oxford: Clarendon Press, 2004.
	
	%\bibitem{LM85}  A. Lubotzky,  A. R.  Magid,  Varieties of representations of finitely generated groups. Mem. Amer. Math. Soc. 58 (1985), no. 336, xi+117 pp.
	
	\bibitem{Lucchini} A. Lucchini, On the number of generators of finite images of free products in finite groups. J.~ Algebra, 245 no. 2, (2001), 552--561. 
		
	
	\bibitem{MelnikovZalesskii} O.V. Mel'nikov and P.A. Zalesski{\u i}, 	Subgroups of profinite groups acting on trees. (Russian)
	Mat. Sb. (N.S.) 135(177) (1988), no. 4, 419–439, 559; translation in	Math. USSR-Sb. 63 (1989), no. 2, 405--424.
	
	
	\bibitem{NS07} N. Nikolov and D. Segal,  On finitely generated profinite groups. I. Strong completeness and uniform bounds. Ann. of Math. (2) 165 (2007),  171--238.
	
	%\bibitem{PP15}  D. Puder,  O. Parzanchevski, Measure preserving words are primitive. J. Amer. Math. Soc. 28 (2015),  63--97.
	
	%\bibitem{Re18} A. Reid, Profinite rigidity,   The Proceedings of the 2018 I.C.M. Rio De Janeiro, Vol. 2 1211--1234. 
	
	% \bibitem{Ribesprofinitetree}   Ribes, Luis Profinite graphs and groups. Ergebnisse der Mathematik und ihrer Grenzgebiete. 3. Folge. A Series of Modern Surveys in Mathematics [Results in Mathematics and Related Areas. 3rd Series. A Series of Modern Surveys in Mathematics], 66. Springer, Cham, 2017. 
	
	\bibitem{PP15}  D. Puder and  O. Parzanchevski, Measure preserving words are primitive. J. Amer. Math. Soc. 28 (2015),  63--97.
	
	\bibitem{Ribesprofinitetree} L.  Ribes,  Profinite graphs and groups. Ergebnisse der Mathematik und ihrer Grenzgebiete. 3. Folge. A Series of Modern Surveys in Mathematics [Results in Mathematics and Related Areas. 3rd Series. A Series of Modern Surveys in Mathematics], 66. Springer, Cham, 2017. 
	
	
	\bibitem{RibesZalesskii}
	L.~ Ribes and P.~Zalesskii. \textit{Profinite groups}. Ergebnisse der Mathematik und ihrer Grenzgebiete. 3. Folge. A Series of Modern Surveys in Mathematics, Vol. 40, 2nd edition, Springer-Verlag, Berlin, 2010. 
	
	 \bibitem{Roseblade73} J.~E.~ Roseblade,   Group rings of polycyclic groups. J. Pure Appl. Algebra 3 (1973), 307--328. 
	

	\bibitem{Roseblade} \rule[0.48ex]{3em}{0.14ex}\space, Applications of the Artin-Rees lemma to group rings, Symposia Mathematica, Vol. XVII (Convegno sui Gruppi Infiniti, INDAM, Rome, 1973) Academic Press, London, 1976, 471--478. 
	
	\bibitem{RR94}  A. Rosenmann and S. Rosset,   Ideals of finite codimension in free algebras and the fc-localization. Pacific J. Math. 162 (1994),  351--371. 

	\bibitem{Serre_trees} J.-P. Serre, Trees. Springer, Berlin, Heidelberg. 1980. 

	\bibitem{Wi18} H.  Wilton, Essential surfaces in graph pairs. J. Amer. Math. Soc. 31 (2018), no. 4, 893--919.

	\bibitem{Wilton_comptes} \rule[0.48ex]{3em}{0.14ex}\space, On the profinite rigidity of surface groups and
	surface words. Comptes Rendus. Math\'ematique. 359 (2021), no.2, 119--122.

\end{thebibliography}
\end{document}